\documentclass[12pt,reqno,twoside]{amsart}
\usepackage{amssymb,amsfonts,cite}
\usepackage{amsthm,amsmath}
\usepackage{hyperref}
\usepackage{color}

\usepackage[left=1 in,top=1 in,right=1 in,bottom=1 in]{geometry}

\newcommand{\beq}{\begin{equation}}
\newcommand{\eeq}{\end{equation}}

\newcommand{\beqs}{\begin{equation*}}
\newcommand{\eeqs}{\end{equation*}}

\newcommand{\ba}{\begin{array}}
\newcommand{\ea}{\end{array}}

\newcommand{\baa}{\begin{array}{rcl}} 
\newcommand{\eaa}{\end{array}}

\newcommand{\beas}{\begin{eqnarray*}}
\newcommand{\eeas}{\end{eqnarray*}}

\newcommand{\bea}{\begin{eqnarray}}
\newcommand{\eea}{\end{eqnarray}}

\newcommand{\bal}{\begin{align}}
\newcommand{\eal}{\end{align}}

\newcommand{\bals}{\begin{align*}}
\newcommand{\eals}{\end{align*}}

\newcommand{\al}{\alpha}
\newcommand{\veps}{\varepsilon}

\newcommand{\lam}{\lambda}

\newcommand{\be}{\beta}
\newcommand{\s}{\sigma}
\newcommand{\ph}{\varphi}
\newcommand{\R}{\ensuremath{\mathbb R}}

\newcommand{\N}{\ensuremath{\mathbb N}}
\newcommand{\Z}{\ensuremath{\mathbb Z}}

\newcommand{\inprod}[1]{\langle{#1}\rangle}

\newcommand{\bds}{\begin{displaystyle}}
\newcommand{\eds}{\end{displaystyle}}

\newcommand{\DA}{\mathcal{D}(A)}
\newcommand{\solnset}{\mathcal{R}}

\def\eqdef{\stackrel{\rm def}{=}}

\newcommand{\bvec}[1]{\mathbf{#1}}
\def\vecx{\bvec x}

\def\vecu{\bvec u}
\def\vecv{\bvec v}
\def\vecw{\bvec w}

\def\veck{\bvec k}

\def\vecl{\bvec l}

\def\vece{\bvec e}
\def\vecm{\bvec m}

\def\varep{\varepsilon}

\def\ddt{\frac{d}{dt}}

\newtheorem{theorem}{Theorem}[section]

\newtheorem{lemma}[theorem]{Lemma}

\newtheorem*{statement}{Statement}

\theoremstyle{remark}
\newtheorem{remark}[theorem]{\bf{Remark}}

\numberwithin{equation}{section}

\def\mD{\mathcal D}

\title{Asymptotic expansion in Gevrey spaces for solutions of Navier-Stokes equations}
\author{Luan T. Hoang$^{1}$ and Vincent R. Martinez$^{2}$}
\address{$^1$Department of Mathematics and Statistics, Texas Tech University,
Box 41042, Lubbock, TX 79409-1042, U.S.A.}
\email{luan.hoang@ttu.edu}
\address{$^2$Mathematics Department, Tulane University, 6823 St. Charles Ave,
New Orleans, LA 70118, U.S.A.}
\email{vmartin6@tulane.edu}

\date{\today}

\begin{document}

\begin{abstract}
In this paper, we study the asymptotic behavior of solutions to the three-dimensional incompressible Navier-Stokes equations (NSE) with periodic boundary conditions and potential body forces. In particular, we prove that the Foias-Saut asymptotic expansion for the regular solutions of the NSE in fact holds in {\textit{all Gevrey classes}}.  This strengthens the previous result obtained in Sobolev spaces by  Foias-Saut.  By using the Gevrey-norm technique of Foias-Temam, the proof of our improved result simplifies the original argument of Foias-Saut, thereby, increasing its adaptability to other dissipative systems. Moreover, the expansion is extended to all Leray-Hopf weak solutions.
\end{abstract}

\maketitle

\pagestyle{myheadings}\markboth{L. T. Hoang and V. R. Martinez}
{Asymptotic Expansion  in Gevrey 
Spaces for Solutions of Navier-Stokes Equations}


\section{Introduction and Main result}\label{Introsec}
The Navier-Stokes equations (NSE) play an essential role in understanding fluid mechanics.
Their long-time dynamics still pose great challenges in both mathematics and physics.
This paper is focused on the asymptotic analysis of solutions to the NSE with periodic boundary conditions in the particular case where the body force is potential.
In this situation, it is elementary to show  that the solution decays exponentially when time is large.
However, to quantify the decay rate precisely is a more difficult problem.
Dyer and Edmunds \cite{DE1968} were the first to obtain an exponential lower bound for non-trivial solutions.
Later, Foias and Saut proved that in bounded or periodic domains the regular, non-trivial solutions of the NSE decay exponentially at an \textit{exact} rate which is an eigenvalue of the Stokes operator (see  \cite{FS84a}). 
Remarkably,  they go on to show that the solution in fact admits an asymptotic expansion \cite{FS87} which details its long-time behavior.
This inspired a number of subsequent studies on this expansion, as well as the associated normal form of the NSE, its normalization map, and invariant nonlinear manifolds (cf. \cite{FS91,FLOZ1,FLOZ2,FLS1} and references therein).
Applications of the expansion to statistical solutions of the NSE,  decaying turbulence, and analysis of helicity are obtained in \cite{FLN1,FLN2}. The result is also extended to  Minea's system \cite{Minea}, and to NSE in the whole space $\mathbb R^3$ \cite{KukaDecay2011}.  All of the aforementioned results are established in \textit{Sobolev spaces} leaving the question open whether these results hold in spaces of stronger regularity.  Indeed, it is well-known that solutions of the NSE regularize instantaneously to the real-analytic class (cf. \cite{FT-Gevrey}).  Thus, the analytic Gevrey class presents itself as a natural class to pose the problem of whether or not this asymptotic expansion holds in these spaces as well.  On the other hand, {the original proof in \cite{FS87}, makes use of rather sophisticated estimates which do not appear to be easily reproducible for other dissipative systems.}

In this paper, we prove that the Foias-Saut expansion indeed holds true in all Gevrey spaces (see Theorem \ref{mainthm}).  The Gevrey spaces are much stronger than the Sobolev spaces since they impose \textit{exponential decay} on the high wave-numbers of the solution.  Moreover, Gevrey norms provide extra information on the solution, particularly, on its radius of analyticity in the spatial variable which is of particular importance in the context of turbulence, see e.g. \cite{BF2014,DoeringTiti,FMRTbook, FrischBook}. 
We remark that the technique of using Gevrey norms goes back to \cite{FT-Gevrey} and is essentially an energy method analogous to that developed for Sobolev norms.  It has since become a standard method for establishing higher-order regularity for a large class of equations (cf. \cite{Biswas2, BF2014, BiswasMartinezSilva, LO, MR1844523, MR1758699, OliverTiti2000}).  
We therefore not only strengthen the asymptotic expansion of Foias-Saut, but{, at least for periodic domains,} provide a streamlined and transparent proof of its existence, rendering it adaptable to other dissipative systems.

In order to state our main result precisely, let us prepare some notations and background.

We consider a viscous, incompressible fluid in $\R^3$ with (kinematic) viscosity  $\nu>0$,  velocity vector field $\vecu(\vecx,t)$, scalar pressure $p(\vecx,t)$, and potential body force $(-\nabla\phi(\vecx,t))$, where $\phi(\vecx,t)$ is a given potential.
Here $\vecx\in \R^3$ is the location vector and $t\in\R$ is  time.
The corresponding fluid's dynamics are then described by the  Navier-Stokes equations, which is given as follows:
\begin{align}\label{nse}
\begin{split}
&\bds \frac{\partial \vecu}{\partial t}\eds  + (\vecu\cdot\nabla)\vecu -\nu\Delta \vecu = -\nabla p - \nabla\phi,\\
&\textrm{div } \vecu = 0.
\end{split}
\end{align}
For the initial value problem, it is specified that 
\beq\label{ini}
\vecu(\vecx,0) = \vecu^0(\vecx),
\eeq 
where  $\vecu^0(\vecx)$ is the given initial velocity field.

We focus on $L$-periodic solutions $(\vecu,p)$, where $\vecu$ has zero average over the domain $\Omega=(-L/2, L/2)^3$, $L>0$, corresponding to $\vecu^0$, which is also assumed to be $L$-periodic with zero average.  Indeed, we may see from \eqref{nse} that the zero-average condition is preserved by the evolution.  Here, a function $f(\vecx)$ is $L$-periodic if
\beqs
f(\vecx+L\vece_j)=f(\vecx)\quad \textrm{for all}\quad \vecx\in \R^3,
\quad j=1,2,3,\eeqs
where  $\{\vece_1,\vece_2,\vece_3\}$ is the standard basis of $\R^3$, and has  zero average over $\Omega$ if 
\beq \label{Zacond} 
\int_\Omega f(\vecx)d\vecx=0.
\eeq
For our purposes, we assume that the potential $\phi(\vecx,t)$ is also $L$-periodic for all $t\ge 0$.

We recall that \eqref{nse} satisfies the following scaling law: 
$$(\vecu_\lam,p_\lam)(\vecx,t)=(\lam\vecu(\lam\vecx, \lam^2t),\lam^2 p(\lam \vecx, t))$$ is a solution of \eqref{nse}, for all $\lam>0$, if $(\vecu, p)$ is a solution of \eqref{nse}.  Thus, by rescaling the spatial and time variables, we may assume throughout, without loss of generality, that  $L=2\pi$ and $\nu =1$.

Let $\inprod{\cdot,\cdot}$ and $|\cdot
|$ denote the inner product and norm in
$L^2(\Omega)^3$, that is,
$$
\inprod{u,v}=\int_\Omega \vecu(\vecx)\cdot \vecv(\vecx) d\vecx
\quad\text{and}\quad  |u|
        =\inprod{u,u}^{1/2}\quad\text{for functions } u=\vecu(\cdot),\ v=\vecv(\cdot).$$
We note that the  notation  $|\cdot|$ is also used to  denote  the Euclidean length of vectors in $\R^3$, but any apparent ambiguity will be clarified by the context.

Recall that $H^m(\Omega)$ with $m=0,1,2,\ldots$ are the Sobolev spaces of
functions on $\Omega$ that have distributional derivatives up to order $m$ belonging to $L^2(\Omega)$.

Let $\mathcal{V}$ be the set of all $L$-periodic trigonometric polynomial vector fields which are divergence-free and   has  zero average over $\Omega$.  
Define
$$H, \text{ resp. } V\ =\text{ closure of }\mathcal{V} \text{ in }
L^2(\Omega)^3, \text{ resp. } H^1(\Omega)^3.$$

We will let $\mathcal{P}$ denote the orthogonal projection in $L^2(\Omega)^3$ onto $H$.

The Stokes operator $A$ with domain
$\DA=V\cap H^2(\Omega)^3$ is defined by
\beqs A\vecu = - \mathcal{P}\Delta \vecu \quad \textrm{for all}\quad \vecu\in\DA.
\eeqs

Note that since we are working with periodic boundary conditions, we simply have $A=-\Delta$ on $\DA$.

Thanks to the zero-average condition \eqref{Zacond}, the norm $\|\vecu\|\eqdef |\nabla \vecu|$ for $\vecu\in V$  is equivalent to the standard $H^1$-norm, and the norm   $|A\vecu|$ for  $\vecu\in \DA$  is equivalent to the standard $H^2$-norm.

It is known that in the setting above, the spectrum  of the Stokes operator, $A$, is 
$$\sigma(A)=\{\lambda_j:j\in \N\},$$
where $\lambda_j$ is strictly increasing in $j$, and each is an  eigenvalue of $A$ with  $\lambda_j=|\veck|^2$ for some $\veck\in\Z^3\setminus \{0\}$.
Observe that the additive semigroup generated by $\sigma(A)$ is simply the set $\N$ of  natural numbers. 

If $n\in\s (A)$, we define  $R_n$ to be the orthogonal projection in $H$ onto the eigenspace of $A$ corresponding to $n$. In case $n\notin \s (A)$,  set  $R_n=0$.  
For $n\in \N$, define $P_n=R_1+R_2+ \cdots +R_n$.

For $\alpha,\sigma \in \R$ and  $\vecu=\sum_{\veck\ne 0} 
\hat \vecu(\veck)e^{i\veck\cdot \vecx}$, define
$$A^\alpha \vecu=\sum_{\veck\ne 0} |\veck|^{2\alpha} \hat \vecu(\veck)e^{i\veck\cdot 
\vecx},$$
$$A^\alpha e^{\sigma A^{1/2}} \vecu=\sum_{\veck\ne 0} |\veck|^{2\alpha}e^{\sigma 
|\veck|} \hat \vecu(\veck)e^{i\veck\cdot 
\vecx},$$
where $\hat{\vecu}(\veck)$ denotes the Fourier coefficient of $\vecu$ at wavenumber $\veck$. We then define the  Gevrey classes by
\beqs
G_{\alpha,\sigma}=\mD(A^\alpha e^{\sigma A^{1/2}} )\eqdef \{ \vecu\in H: |\vecu|_{\alpha,\sigma}\eqdef |A^\alpha 
e^{\sigma A^{1/2}} \vecu|<\infty\}.
\eeqs
Then the domain of $A^\al$ is $\mD(A^\alpha)=G_{\alpha,0}$.
Note that $\mD(A^0)=H$,  $\mD(A^{1/2})=V$, and  $\|\vecu\|=|A^{1/2}\vecu|$ for $\vecu\in V$.

We  define the bilinear mapping associated with
the nonlinear term in the Navier-Stokes equations  by
\beqs
B(\vecu,\vecv)=\mathcal{P}(\vecu\cdot \nabla \vecv) \quad \textrm{for all}\quad \vecu,\vecv\in\DA.
\eeqs

For convenience, we will denote $u(t)=\vecu(\cdot,t)$ from now on.  Thus, by applying the Leray projection, $\mathcal{P}$, to  \eqref{nse} and \eqref{ini}, the initial value problem for NSE may be re-written in the functional form as
\beq\label{fctnse}
\frac{du(t)}{dt} + Au(t) +B(u(t),u(t))=0\quad \forall t> 0,
\eeq
with the initial data 
\beq\label{uzero} 
u(0)=u^0\in H.
\eeq
(See e.g. \cite{CFbook} or  \cite{TemamAMSbook} for more details.)

We recall the following local existence theorem \cite{CFbook, LadyFlowbook69,TemamAMSbook}: For any $u^0\in V$ there exists $T\in(0,\infty]$ and a (unique) solution $u(t)$ of \eqref{fctnse} on $(0,T)$ such that $u(t)$ is continuous from $[0,T)$ to $V$ and satisfies  \eqref{uzero}. We call such $u(t)$ a regular solution on $[0,T)$.
Moreover, if  $T_{\rm max}$ is the maximal time of existence for regular solution $u(t)$ and $T_{\rm max}<\infty$, then
\beqs
\lim_{t\to T_{\rm max}^-}\|u(t)\|=\infty.
\eeqs

We denote by $\solnset$ the set of all $u^0\in V$ such that the regular solution $u(t)$ of \eqref{fctnse} and \eqref{uzero} exists on $[0,\infty)$.


For any $u^0\in \solnset$, it is proved in   \cite{FS87} that 
the regular solution $u(t)$ of \eqref{fctnse} and \eqref{uzero} has an asymptotic expansion
\beq \label{expand}
u(t) \sim \sum_{n=1}^\infty q_n(t)e^{-nt},
\eeq
where $q_j(t)$'s are  {unique} polynomials in $t$ 
with values in $\mathcal V$.  This means that for any
$N\in \N$ the remainder  $v_N(t)=u(t)-\sum_{j=1}^N q_j(t)e^{-jt}$
satisfies
\beq\label{cterma}
\|v_N(t)\|_{H^m(\Omega)^3}=\mathcal O\big(e^{-(N+\varepsilon_N)t}\big)\textrm{ as }t\to\infty
\textrm{ for some }\varepsilon_N>0\text{ and all } m=0,1,2,3\ldots,
\eeq

The notation, $\mathcal O(f(t))$, above is defined as 
	\beqs
		\Phi(t)=\mathcal{O}(f(t))\ \text{as $t\to\infty$}\quad \text{if and only if }\exists T,C>0,\ \Phi(t)\leq Cf(t),\ {\forall t>T},
	\eeqs
	where $\Phi$ and $f$ are non-negative scalar quantities.  
We will also make use of the following notation: for $v(t)$ which belongs to $G_{\al,\s}$ eventually, 
	\beqs
		v(t)=\mathcal{O}_{\al,\s}(f(t))\ \text{as $t\to\infty$}\quad \text{if and only if }
		\exists T,C>0,\ |v(t)|_{\al,\s}\leq Cf(t),\ {\forall t>T}.
	\eeqs
We note that the times, $T$, and absolute constants, $C$, may depend on the parameters appearing in $f$.  The dependence on such parameters in the estimates performed below will be indicated as needed.

Our main result is the following improvement.

\begin{theorem}[Main theorem]
\label{mainthm}
The expansion \eqref{expand} holds on any Gevrey space $G_{\alpha,\sigma}$ with 
$\alpha,\sigma >0$. More precisely, for any Leray-Hopf weak solution $u(t)$ of \eqref{fctnse}, there are polynomials $q_n(t)$'s  in $t$ valued in $\mathcal V$ for all $n\in\N$ such that if $\alpha,\sigma >0$ and $N\ge 1$ then
\beq\label{remain}
\Big|u(t)-\sum_{n=1}^N q_n(t)e^{-nt}\Big|_{\alpha,\sigma}= \mathcal O 
\big(e^{-(N+\varep)t}\big)\quad\textrm{as } t\to\infty,\text{ for any }\varep\in(0,1).
\eeq
\end{theorem}

Regarding the Leray-Hopf weak solutions, see e.g. \cite{FMRTbook}.  Some remarks are in order for Theorem \ref{mainthm}:

(a) It suffices to state \eqref{remain} with all $\s>0$ and a fixed $\alpha$, say, $\alpha=0$.
However, we keep the stated form \eqref{remain} for the generality of the Gevrey norms.

(b) In case $u(0)=u^0\in \solnset$, the polynomials  $q_n$'s are uniquely determined by $u^0$. In general, they depend on the solution $u(t)$. Indeed, for a \textit{fixed} weak solution, $u(t)$, the corresponding polynomials, $q_n$'s, are then uniquely determined due to the asymptotic properties of the expansion.

(c) The extension of the Foias-Saut result from regular solutions to Leray-Hopf weak solutions can be useful in the study of turbulence. (See a similar extension for the normalization map of Leray-Hopf weak solutions in \cite{FLN2}.)

\medskip 

Theorem \ref{mainthm} will be proved in Section \ref{expandsec}. Although the proof follows the original scheme in \cite{FS87}, by working directly in Gevrey classes, the need for complicated, recursive estimates in Sobolev spaces for the solution's time derivatives and its higher orders is \textit{eliminated completely}, thereby simplifying the proof considerably in addition to improving  {significantly} the regularity of the expansion.  In particular, since one no longer needs to appeal to particular higher regularity results for the Stokes operator, this approach indicates an avenue for establishing such an expansion to other dissipative systems.  
Let us also point out that the setting of periodic boundary conditions that we consider here is an example of the more difficult case dealt with in \cite{FS87} when there are resonances in the eigenvalues of the Stokes operator.  Our choice of this setting is for the availability of explicit eigenfunctions which are convenient to work with when estimating the Gevrey norms of the bilinear operator $B(u,v)$, see Lemma \ref{nonLem}.

The main observation in proving Theorem \ref{mainthm} is that for each $\s>0$ and $N\ge 1$, the remainder estimate \eqref{remain} is conjectured and then proved, by induction in $N$, to hold true for \emph{all} $\al>0$. This is crucial due to the estimate of the nonlinear mapping $B(u,v)$ which always requires the regularity of one more derivative for $u$ or $v$, see Lemma \ref{nonLem}. However, since the Gevrey norm in $|u|_{0,\s}$ for any $\s>0$ is stronger than \emph{all} Sobolev norms $|A^\al u|$, this obstacle becomes a non-issue.
What remains then in dealing with the weak solutions is that the time of eventual regularity  in the class $G_{\al,\s}$ must be uniform in $\al$, albeit its dependence on $\s$.
By appealing to the Leray energy inequality to enter a small-data regime, we establish this together with asymptotic bounds in these Gevrey norms (Lemma \ref{gev:decay} and Theorem \ref{corsmall}). Note that these bounds are obtained with an exact exponential decay rate, hence, allowing us to deduce the exponential decay rates for the remainders in a straightforward manner.

\section{Basic estimates}\label{Auxsec}

In this section, we derive  estimates for the Gevrey norms of the solutions, particularly, when time is large.
First, we state some basic inequalities.
For all $\al,\s\ge 0$,
\begin{align}\label{PI}
|u|&\le |A^\al u|\quad \text{(Poincar\'e's inequality)},\\
\label{GI}
|u|&\le e^{-\s} |e^{\s A^{1/2}} u|.
\end{align}

When $\s,\al>0$, one has
\beqs 
\max_{x\ge 0} (x^{2\al}e^{-x})=\Big(\frac{2\al}{e}\Big)^{2\al},
\eeqs 
hence
\beq\label{als}
|A^\al e^{-A^{1/2}}u| \le \Big(\frac{2\al}{e}\Big)^{2\al} |u|.
\eeq

Regarding the bilinear mapping $B(u,v)$, we will use the following inequalities which are proved  in Appendix \ref{apex}.

\begin{lemma}
\label{nonLem}
For $\s,\alpha\ge 0$ one has
\beq\label{B1} 
|B(u,v)|_{\al,\s}\le 4^\alpha c_*\Big(|u|_{1/2,\s}^{1/2}\, |u|_{1,\s}^{1/2}\, |v|_{\al+1/2,\s}
+|u|_{\al+1/4,\s} |v|_{1,\s}\Big), 
\eeq
\beq\label{B2} 
|B(u,v)|_{\al,\s}\le 4^\alpha c_*\Big(|u|_{1/2,\s}^{1/2}\, |u|_{1,\s}^{1/2}\, |v|_{\al+1/2,\s}
+|u|_{\al+1/2,\s}\, |v|_{3/4,\s}\Big), 
\eeq
where $c_*>0$ is independent of $\alpha,\sigma$. In particular, if $\alpha\ge 1/2$ then 
\beq\label{AalphaB} 
|B(u,v)|_{\al,\s}\le K^\alpha  |u|_{\al+1/2,\s} \, |v|_{\al+1/2,\s} \quad \forall u,v\in G_{\alpha+1/2,\s},  \eeq
where $K=4(\max\{2c_*,1\})^2$.
\end{lemma}

Here afterward, we denote $C_\al=K^\al$ with $K\ge 1$  is the  constant in \eqref{AalphaB}.

We start by establishing uniform-in-time estimates when initial data is small.

\begin{lemma}\label{gev:decay}
Let $\al\ge 1/2$ and $\delta\in(0,1)$. If $u^0\in D(A^\al)$ satisfies
	\begin{align}\label{small1}
		|A^\al u^0|<\frac\delta{2C_\al},
	\end{align}
then the regular solution $u(t)$ of \eqref{fctnse} and \eqref{uzero} exists on $[0,\infty)$ and satisfies  $u(t)\in D(A^\beta)$ for all $\beta>0$ and $t>0$.
Moreover, for any $\s>0$
\beq\label{cl}
|u(t)|_{\al,\s}\le e^{-(1-\delta)t}|A^\alpha u^0| \quad \forall\  t\ge 4\s/\delta.
\eeq
\end{lemma}
\begin{proof}
The following calculations are formal but can be made rigorous by using solutions of the Galerkin approximations of \eqref{fctnse}, and the standard passage to the limit, see e.g. \cite{TemamAMSbook}.

Let $\veps>0$,  and $\varphi\in C^\infty(\R)$ such that $\varphi(t)=0$ for $t\le 0$, $\varphi(t)>0$ for $t>0$, 
$\varphi(t)=\sigma$ for $t\ge 2\sigma/\veps$, and $0<\varphi'(t)\le \veps$ for $t\in(0,2\sigma/\veps)$. Then 
\beq\notag
\frac{d}{dt} (A^\alpha e^{\varphi A^{1/2}}u)+ A^{\alpha+1} e^{\varphi A^{1/2}}u 
=-A^\alpha e^{\varphi A^{1/2}}B(u,u)+\varphi'(t)A^{\alpha+1/2} e^{\varphi 
A^{1/2}}u.
\eeq
Taking the inner product of the equation with $A^\alpha e^{\varphi A^{1/2}}u$, we have
\begin{multline*}
\frac12\frac{d}{dt} |A^\alpha e^{\varphi A^{1/2}}u|^2+ |A^{\alpha+1/2} 
e^{\varphi A^{1/2}}u|^2\\
 =- \langle A^\alpha e^{\varphi A^{1/2}}B(u,u),A^\alpha e^{\varphi 
A^{1/2}}u\rangle+\varphi'(t)\langle A^{\alpha+1/2} e^{\varphi A^{1/2}}u,A^\alpha e^{\varphi 
A^{1/2}}u\rangle.
\end{multline*}
Applying Cauchy-Schwarz inequality and \eqref{AalphaB} to the terms on then right-hand side yields
\begin{align}\label{raw0}
\frac12\frac{d}{dt} |u|_{\al,\varphi}^2+ |A^{1/2}u|_{\al,\varphi}^2
 \le C_\alpha | A^{1/2}u|_{\al,\varphi}^2 |u|_{\al,\varphi}+\varphi'(t)|A^{1/2}u|_{\al,\varphi} |u|_{\al,\varphi}.\end{align}
Note that $C_\al\geq4$.  Then
\begin{align}\label{raw1}
\frac12\frac{d}{dt} |u|_{\al,\varphi}^2+ |A^{1/2}u|_{\al,\varphi}^2
\le (C_\al |u|_{\al,\varphi}+\varep)|A^{1/2}u|_{\al,\varphi}^2.
\end{align}
Letting $\veps=\delta/2$, we obtain
\beq\label{raw}
\frac12\frac{d}{dt} |u|_{\al,\varphi}^2+ (1-\frac\delta2-C_\alpha|u|_{\al,\varphi}) |A^{1/2}u|_{\al,\varphi}^2\le 0.
\eeq

We claim that 
\beq\label{claim} 
C_\al|u(t)|_{\al,\ph(t)} \le \delta/2 \quad\forall t\ge 0.
\eeq

Suppose \eqref{claim} is not true, then by \eqref{small1}, there is $T\in(0,\infty)$ such that
\begin{align}\label{smallall} 
C_\al|u(t)|_{\al,\ph(t)} &< \delta/2 \quad\forall t\in[0,T), \\
\label{uT} 
C_\al|u(T)|_{\al,\ph(T)} &= \delta/2. 
\end{align}

By \eqref{raw} and \eqref{smallall}, we have for $t\in(0,T)$ that
\beq\label{dsmall}
\frac12\frac{d}{dt} |u|_{\al,\varphi}^2+ (1-\delta)| A^{1/2}u|_{\al,\varphi}^2\le 0\quad \forall t\in(0,T).
\eeq
Hence
\beqs
|u(t)|_{\al,\varphi(t)}\le |u^0|_{\al,\varphi(0)}=|A^\al u^0|\quad\forall t\in(0,T). 
\eeqs

Passing $t\nearrow T$ gives
\beqs
|u(T)|_{\al,\varphi(T)}\le |A^\al u^0|<\frac\delta{2 C_\al}, 
\eeqs
which contradicts \eqref{uT}. Therefore, \eqref{claim} holds true. 
Consequently, $u(t)$ is a regular solution on $[0,\infty)$.

For $t>0$, we have $\varphi(t)>0$, then for any $\beta>0$, applying inequality \eqref{als} with $\al=\beta$ and $\s=\varphi(t)$, we have $u(t)\in \mD(A^\beta)$.

As a consequence of \eqref{claim}, differential inequality \eqref{dsmall} now holds for all $t>0$. By Poincar\'e's  inequality,
\beqs
\frac12\frac{d}{dt} |u|_{\al,\varphi}^2+ (1-\delta)| u|_{\al,\varphi}^2\le 0\quad \forall t>0.
\eeqs
Then by  Gronwall's inequality,
\beq\label{decay0}
|u(t)|_{\al,\varphi(t)}^2\le e^{-2(1-\delta)t}|u^0|_{\al,\varphi(0)}^2=e^{-2(1-\delta)t}|A^\alpha u^0|^2\quad\forall 
t>0.
\eeq

When $t\ge 4\s/\delta$,   $\varphi(t)=\s$, then we obtain  \eqref{cl} from \eqref{decay0}.
\end{proof}

Next, we improve the exponential decay rate in \eqref{cl} from $e^{-(1-\delta)t}$ to $e^{-t}$.

\begin{theorem}\label{corsmall}
Assume $\al\ge 1/2$ and $u^0\in \mD(A^{\al+1/2})$ satisfy 
\beq\label{doublesmall}
|A^{\al+1/2} u^0|< \frac1{12 C_{\al+1/2}}.
\eeq
Let $u(t)$ be the regular solution of \eqref{fctnse} and \eqref{uzero} on $[0,\infty)$.
Then one has for any  $\s>0$ and all $t\ge 24\s$ that
\beq\label{d1}
|u(t)|_{\al,\s}
\le  \sqrt 2e^{4\sigma}| A^{\al+1/2}  u^0| e^{-t}
\eeq
and, consequently, 
\beq\label{decaymore}
|u(t)|_{\al,\s}
\le \frac{e^{4\sigma}}{6\sqrt 2C_{\alpha+1/2}}  e^{-t} .
\eeq
\end{theorem}
\begin{proof}
Take $\delta=1/6$, then by \eqref{doublesmall},	
$C_{\al+1/2}|A^{\al+1/2} u^0|<\delta/2$, thus condition  \eqref{small1}  is met for  $\al\to\al+1/2$. 
We apply in Lemma \ref{gev:decay} for $\al\to \al+1/2$. Then the regular solution $u(t)$ exists on $[0,\infty)$.
For $t\ge t_0\eqdef  4\s/\delta=24\s$, estimate \eqref{cl} for $\al\to \al+1/2$ gives 
\beq\label{cle}
|A^{1/2}u(t)|_{\al,\s}=|u(t)|_{\al+1/2,\s}\le 
e^{-5t/6}\|A^\alpha u^0\|.
\eeq

For $t\ge t_0$, we have $\varphi'(t)=0$, and, by combining \eqref{raw0} with \eqref{cle}, obtain
\begin{align*}
\frac12\frac{d}{dt} |u|_{\al,\s}^2+ |A^{1/2}u|_{\al,\s}^2
&\le C_\alpha |A^{1/2}u|_{\al,\s}^2 |u|_{\al,\s}
\le C_\alpha |A^{1/2}u|_{\al,\s}^3
\le C_\alpha e^{-5t/2}  \| A^\al u^0\|^3.
\end{align*}

By Poincar\'e's and Gronwall's inequalities, we have for $t\ge t_0$ that
\begin{align*}
|u(t)|_{\al,\s}^2
\le e^{-2(t-t_0)} |u(t_0)|_{\al,\s}^2 
+4C_\alpha e^{-2t} \| A^\al u^0\|^3.\end{align*}

Estimating the first norm on the right-hand side by \eqref{cle} with $t=t_0$ yields
\begin{align*}
|u(t)|_{\al,\s}^2
&\le e^{-2(t-t_0)} e^{-5t_0/3} \|A^\alpha  u^0\|^2  
+4C_\alpha e^{-2t}  \| A^\al  u^0\|^3 \\
&\le e^{-2t} (e^{8\sigma}+4C_\alpha \| A^\al  u^0\|)\| A^\al  u^0\|^2.
\end{align*}
 Using \eqref{doublesmall} to estimate  $\| A^\al  u^0\|$ between the parentheses, we obtain 
\beqs
|u(t)|_{\al,\s}^2
\le e^{-2t} (e^{8\sigma} +1/3) \| A^\al  u^0\|^2\le e^{-2t} 2e^{8\sigma}\| A^\al  u^0\|^2,
\eeqs
and inequality \eqref{d1} follows.
Using \eqref{doublesmall}  again in \eqref{d1} yields \eqref{decaymore}.
\end{proof}

The following gives estimates for  the Gevrey norms of  Leray-Hopf weak solutions  with the optimal exponential decay  for large time.

\begin{theorem}\label{decay2}
Let $u^0\in H$ and $u(t)$ be a Leray-Hopf weak solution of \eqref{fctnse} and \eqref{uzero} on $[0,\infty)$. 
For any $\s>0$,   there exist $T=T(\s,|u^0|)>0$ and $D_\s=D_\s(|u^0|)$ such that
	\beq\label{cl1}
|u(t)|_{1/2,\s+1}\le D_\s e^{-t}\quad \forall t\ge T.
	\eeq
Consequently, for any $\alpha\ge 0$ there is $D_{\al,\s}=D_{\al,\s}(|u^0|)>0$ such that 
\beq\label{cl10}
|u(t)|_{\al+1/2,\s}\le D_{\al,\s} e^{-t}\quad \forall\ t\ge T.
\eeq
The values of $T$, $D_\s$ and $D_{\al,\s}$ can be explicitly specified as in \eqref{T}, \eqref{Ds} and \eqref{Das}.
\end{theorem}

\begin{proof}
Taking inner product of \eqref{fctnse} with $u$ and using the orthogonality property 
$$\langle B(u,u),u\rangle=0$$ we have
\beq\label{deq}
\frac12\ddt |u|^2+\|u\|^2=0.
\eeq
Then for $t\ge 0$
\beqs
|u(t)|^2\le e^{-2t}|u^0|^2,
\eeqs
and for any $T_0\ge 0$, by integrating \eqref{deq} from $T_0$ to $T_0+1$, we have 
\beq\label{energy}
\int_{T_0}^{T_0+1}\|u(\tau)\|^2d\tau \le \frac12|u(T_0)|^2\le 
\frac{e^{-2T_0}}2|u^0|^2.
\eeq

The above calculations are valid for regular solutions. For Leray-Hopf weak solutions, the energy inequality \eqref{energy} holds for $T_0=0$ and also almost all $T_0\in (0,\infty)$.

Take $T_0\ge 0$ such that \eqref{energy} holds and 
\beq\label{T0def} 
(\ln(4C_{1/2}|u^0|))^+ < T_0 <(\ln(4C_{1/2}|u^0|))^+ +1,\eeq
which implies
\beq \label{T0cond}
e^{-T_0}|u^0|<1/(4C_{1/2}).
\eeq
By \eqref{energy}, there exists $t_*\in(T_0,T_0+1)$ such that 
\beq\label{tst}
|A^{1/2} u(t_*)|< e^{-T_0}|u^0| < 1/(4C_{1/2}).
\eeq

Applying Lemma  \ref{gev:decay} to $\s=1$, $\al=\delta=1/2$ with initial time $t_*$, then  \eqref{cl} implies 
for all $t\ge T_1\eqdef t_*+8$ that
\begin{align*}
|u(t)|_{1/2,1}
&\le |A^{1/2} u(t_*)|e^{-\frac12(t-t_*)}
\le e^{-T_0}|u^0| e^{\frac12(-t+T_0+1)}
=|u^0| e^{\frac12(-t-T_0+1)}.
\end{align*}
Then for all $t\ge T_1$,
\begin{align}
|Au(t)|
&=|A^{1/2}e^{- A^{1/2}}\big(A^{1/2}e^{A^{1/2}}u(t)\big)| \notag \\
\text{by \eqref{als}}&\le e^{-1} |u(t)|_{1/2,1} 
\le e^{-1}|u^0| e^{(-t-T_0+1)/2}=|u^0| e^{-(t+T_0+1)/2}.\label{i2s}
\end{align}

Let $T_2 \ge T_1$ such that
\beq\label{T2cond}
|u^0|e^{-(T_2+T_0+1)/2}<1/(12C_1).
\eeq
Then \eqref{i2s} gives
\beqs
|Au(T_2)|<1/(12C_1).
\eeqs

Applying Theorem \ref{corsmall}  to initial time $T_2$, $\al=1/2$ and $\s\to\s+1$, we have from \eqref{decaymore} that if  
$t\ge T\eqdef T_2+24(\s+1)$ then
\beqs
|u(t)|_{1/2,\s+1}
\le e^{-(t-T_2)}  \frac{e^{4(\sigma+1)}}{6\sqrt 2 C_1}
=  \frac{e^{T_2+4(\sigma+1)}}{6\sqrt 2 C_1} e^{-t} .
\eeqs
Thus, we obtain \eqref{cl1} with
\beq\label{Ds0}
D_\s=\frac{ e^{T_2+4\s+4}}{6\sqrt 2 C_1}.
\eeq

Using \eqref{als} again, we have for all $t\ge T$ that
\begin{align*}
|u(t)|_{\al+1/2,\s}
&=|A^\al e^{- A^{1/2}}\big(A^{1/2}e^{(\sigma+1) A^{1/2}}u(t)\big)| 
\le (2\al/e)^{2\al} |u(t)|_{1/2,\s+1} 
&\le (2\al/e)^{2\al} D_\s e^{-t},
\end{align*}
which proves \eqref{cl10} with 
\beq\label{Das}
D_{\al,\s}=(2\al/e)^{2\al} D_\s .
\eeq

For dependence of $T$ and $D_\s$ on $\s$ and $|u_0|$,  we note from \eqref{T0def} that 
$$ t_* < T_0+1<(\ln(4C_{1/2}|u^0|))^+ +2.$$
Then $T_1$ satisfies
$$
T_1=t_*+8<10+(\ln(4C_{1/2}|u^0|))^+ \le 10+(\ln (12C_1|u^0|) )^+.$$
For $T_2$, we  need $T_2\ge T_1$ and, from \eqref{T2cond}, 
$$T_2> 2(\ln (12C_1|u^0|) )^+ -T_0-1.$$
Note that the last lower bound satisfies
\begin{align*}
&2(\ln (12C_1|u^0|) )^+ -T_0-1
< 2(\ln (12C_1|u^0|) )^+-(\ln(4C_{1/2}|u^0|))^+ -1\\
&< 2(\ln (12C_1|u^0|) )^+-(\ln(12C_1|u^0|))^+ 
= (\ln (12C_1|u^0|) )^+ .
\end{align*}
We select
$T_2=10+(\ln(12C_1|u^0|))^+$.
Consequently, 
\beq\label{T}
T=T_2+24(\s+1) =24\s+34+(\ln(12C_1|u^0|))^+,
\eeq
and, from \eqref{Ds0}, 
\begin{align}\label{Ds}
D_\s&=\frac{e^{4\s+14} \max\{12C_1|u^0|,1\} }{6\sqrt 2 C_1}
=\sqrt 2  e^{4\s+14}  \max\{|u^0|,1/(12C_1)\} .
\end{align}
The dependence of $D_{\al,\s}$ on $\al$, $\s$, $|u^0|$ is clear from \eqref{Das} and \eqref{Ds}. The proof is complete.
\end{proof}

\section{Proof of the main theorem}\label{expandsec}

We prove Theorem \ref{mainthm} in this section. 
Let  $u(t)$ be a Leray-Hopf weak solution of \eqref{fctnse}. 

Let $\s>0$ be fixed. Since the asymptotic expansion \eqref{expand} only involves the asymptotic behavior of $u(t)$ as $t\to\infty$, then 
by Theorem \ref{decay2} and shifting the initial time to $T=T(\s,|u_0|)$, which is independent of $\al$, we assume, at the moment, that
$u(t)\in G_{1/2,\s+1}$ for all $t\in[0,\infty)$ and satisfies NSE on $[0,\infty)$ in $G_{\al,\s}$ for all $\al>0$.
By \eqref{cl10}, for any $\alpha>0$ and $t\ge 0$,
\beq\label{Aas}
|u(t)|_{\al+1/2,\s} \le M_\al e^{-t}
\eeq
and, consequently by \eqref{AalphaB},
\beq\label{Bas}
|B(u(t),u(t))|_{\al,\s} \le B_\al e^{-2t},
\eeq
where $M_\al$ and $B_\al$ are positive constants depending on $\al$, $\s$ and $|u^0|$.

\begin{lemma}\label{STsigma}
For all $N\in\N$, the following statement ($\mathcal T_N$) holds true.
\begin{statement}[$\mathcal T_N$] 
There are polynomials $q_n(t)$ in $t$, for $1\le n\le N$,  with values in $\mathcal V$ such that the functions $u_n(t)\eqdef e^{-nt}q_n(t)$ for $1\le n\le N$ have the following properties.
\begin{enumerate}
\item
The remainder $v_N(t)=u(t)-\sum_{n=1}^N u_n(t)$ satisfies for any $\alpha>0$  that
\beq\label{err}
\big|v_N(t)\big|_{\al,\s}=\mathcal O (e^{-(N+\varep_*)t}) \quad \text{as }t\to\infty \text{ with } \varep_*=1/2.
\eeq

\item There is $\xi_1\in R_1H$ such that $q_1(t)=\xi_1$ for all $t\in\R$.

\item For $2\le n\le N$,
\beq\label{um:eqn}
\ddt u_n+Au_n+\sum_{\stackrel{1\le m,j\le n-1}{m+j=n}} B(u_m,u_j)=0\quad\forall t\in\R.
\eeq
\end{enumerate}
\end{statement}
\end{lemma}

\begin{proof}
First, we make the following remarks on Statement ($\mathcal T_N$):

(a) Obviously,  $u_1'(t)+u_1(t)=0$, and since $u_1(t)\in R_1H$, we have $u_1=Au_1$ and hence
\beq\label{du1}
\ddt u_1+Au_1=0\quad\forall t\in\R.
\eeq

(b) Equation \eqref{um:eqn} is posed on a finite-dimensional space.  Hence, it is a system of ordinary differential equations and no norm needs to be indicated.

\medskip

Now we prove  ($\mathcal T_N$) by induction in $N$. 

\medskip
\noindent\textbf{I. Base case ($N=1$).}
We construct a constant $\xi_1\in R_1H$ such that $v_1\eqdef u-u_1$ with $u_1\eqdef\xi_1e^{-t}$ satisfies \eqref{err}.  As we remarked above, clearly,  \eqref{du1} is satisfied in this situation.

To construct $\xi_1$, we apply the projection $R_1$ to equation \eqref{fctnse} to get
	\beq\notag
		\frac{d}{dt}R_1u(t)+R_1u(t)+R_1B(u(t),u(t))=0.
	\eeq
Then by Gronwall's inequality
\beq\label{t12}
e^{t}R_1u(t)=R_1u^0-\int_{0}^{t} e^\tau R_1B(u(\tau),u(\tau))\ d\tau.
\eeq
Thanks to \eqref{Bas}, the improper integral $\int_0^\infty e^\tau R_1B(u(\tau),u(\tau))d\tau$  exists and belongs to $R_1H$.
\begin{align}\label{37}
	\xi_1\eqdef\lim_{t\to\infty} e^t R_1u(t)=R_1u^0-\int_0^\infty e^\tau R_1B(u(\tau),u(\tau))\ d\tau \quad  \text{belongs in}\ R_1H.
\end{align}

Solving for $R_1u^0$, we may then rewrite \eqref{t12} as
	\beq\label{38}
		e^{t}R_1u(t)
		=\xi_1+\int_{t}^\infty e^{\tau}R_1B(u(\tau),u(\tau))\ d\tau.
	\eeq
Define the constant polynomial $q_1(t)\eqdef \xi_1$.
Then from \eqref{Bas} and \eqref{38}, we have for $t>0$ that 
\begin{align*}
|e^{t}R_1u(t)-q_1(t) |_{\al,\s}
&\le \int_t^\infty e^{\tau}|R_1B(u(\tau),u(\tau))|_{\al,\s}\ d\tau\\
&=\int_t^\infty e^{\tau} e^\s |B(u(\tau),u(\tau))|\ d\tau
 \le C\int_t^\infty e^{\tau}e^{-2\tau}\ d\tau \le Ce^{-t}
\end{align*}
for some $C>0$.
Multiplying the preceding inequality by $e^{-t}$ gives 
\beq\label{decayR1}
|R_1u(t)-u_1(t)|_{\al,\s}=\mathcal{O}(e^{-2t}),\quad \text{where}\quad  u_1(t)\eqdef e^{-t}q_1(t).
\eeq
Observe that
	\beqs
		v_1=u-u_1=(I-R_1)u+(R_1u-u_1).
	\eeqs
Then, in light of \eqref{decayR1}, to show \eqref{err}, it suffices to consider $(I-R_1)u$.

Applying the complementary projection $(I-R_1)$ to \eqref{fctnse} gives
	\beqs
		\frac{d}{dt}(I-R_1)u+ A(I-R_1)u+(I-R_1)B(u,u)=0.
	\eeqs
Taking the scalar product of this equation with $A^{2\al} e^{2\sigma A^{1/2}}u$, applying  inequalities \eqref{Aas} and \eqref{Bas}, we obtain for all $t>0$ that
	\begin{multline}\label{Q1}
		\frac{1}2\frac{d}{dt}|A^{\alpha}e^{\sigma A^{1/2}}(I-R_1)u|^2+\|A^\al e^{\sigma A^{1/2}}(I-R_1)u\|^2
		\leq |\langle A^{\al}e^{\sigma A^{1/2}}B(u,u),A^{\al}e^{\sigma A^{1/2}}(I-R_1)u\rangle |\notag\\
		\leq |A^\al e^{\sigma A^{1/2}}B(u,u)|\cdot  |A^\al e^{\sigma A^{1/2}}u|
		\le D_\al e^{-3t}
	\end{multline}
	for  $D_\al=B_\al M_\al >0$.
Since $\|A^{\alpha}e^{\sigma A^{1/2}}(I-R_1)u\|^2\ge 2|A^{\alpha}e^{\sigma A^{1/2}}(I-R_1)u|^2$, it follows that
\beqs		
\frac12 \frac{d}{dt}|(I-R_1)u|_{\al,\s}^2+2|(I-R_1)u|_{\al,\s}^2
\le  D_\al e^{-3t}\quad \forall\ t>0.
\eeqs
By Gronwall's inequality 
\beqs
|(I-R_1)u|_{\al,\s}^2\le e^{-4t} |(I-R_1)u^0|_{\al,\s}^2 + 2D_\al e^{-3t}\quad\forall t>0.
\eeqs	
Therefore,
\beq\label{decayQ1}
|(I-R_1)u(t)|_{\al,\s}=\mathcal O(e^{-3t/2})\quad \text{as }t\to\infty.
\eeq	

We conclude  from \eqref{decayR1} and \eqref{decayQ1} that $|v_1(t)|_{\al,\s}=\mathcal O(e^{-3t/2})$,
thus proving statement ($\mathcal T_1$).

\medskip
\noindent\textbf{II. Induction step.}
Let $N\ge 1$ be any natural number. Suppose that the statement ($\mathcal T_N$) is true. 

Let $q_n(t)$, $u_n(t)$, $v_N(t)$ and $\varep_*$ be as in ($\mathcal T_N$).  Denote $\bar u_N(t)=\sum_{n=1}^N u_n(t)$.

Then $v_N(t)=u(t)-\bar{u}_N(t)$, 
and for any $\beta>0$ 
\beq\label{ind:hyp1}
|v_N(t)|_{\beta,\s}=\mathcal O(e^{-(N+\varep_*)t})\text{ as } t\to\infty.
\eeq

Also, since $q_j\in\mathcal{V}$, there are $s_j\in\N$ such that $q_j\in P_{s_j}H$ for $j=1,\ldots, N$.
Thus, there exists $s_{N+1}>\max\{s_n:1\le n\le N\}$ such that
$B(q_m(t),q_j(t))$ are polynomials of $t$ valued in $P_{s_{N+1}}H$ for all  $1\le m,j\le N$.
Lastly, let us fix $\al\ge 1/2$.   We organize the induction step into several parts.

\medskip
\noindent\textit{II.1.~Evolution of $v_N$.} 
From equation \eqref{fctnse} for $u(t)$, and equations \eqref{du1} and \eqref{um:eqn} for $u_n(t)$'s with $1\le n\le N$, we have
\begin{align*}
\frac d{dt}v_N+Av_N
&= -B(u,u)+\sum_{m+j\le N} B(u_k,u_j)\\
&= -B(u,u)+ B(\bar u_N,\bar u_N)-\sum_{\stackrel{1\le m,j\le N}{m+j> N}} B(u_m,u_j)\\
&=-B(v_N,u)-B(\bar u_N,v_N)-\sum_{m+j=N+1} B(u_m,u_j)-\sum_{\stackrel{1\le m,j\le N}{m+j\ge  N+2}} B(u_m,u_j).
\end{align*}
Therefore,
\beq\label{vNeq}   
\frac d{dt}v_N+Av_N +\sum_{m+j=N+1} B(u_m,u_j)=h_N,
\eeq
where
\beq\label{hdef}
h_N(t)=-B(v_N,u)-B(\bar u_N,v_N)-\sum_{\stackrel{1\le m,j\le N}{m+j\ge  N+2}} B(u_m,u_j).
\eeq 

\medskip
\noindent\textit{II.2.~Bounds for $h_N$.}
We claim that
\beq\label{hNo}
h_N(t)=\mathcal O_{\al,\s}(e^{-(N+1+\varep_*)t}).
\eeq
Hence, there exist $T_\al>0$ and $H_{N,\al}>0$ such that 
\beq\label{hNp}
|h_N(t)|_{\al,\s}\le H_{N,\al} e^{-(N+1+\varep_*)t},\quad\forall t\ge T_\al.
\eeq

First, observe that by the induction hypothesis, \eqref{ind:hyp1} is satisfied with $\be=\al+1/2$, so that
\beq\label{vN2}
 v_N(t)=\mathcal O_{\al+1/2,\s}(e^{-(N+\varep_*)t}).
\eeq
Then, by inequality \eqref{AalphaB} and the facts \eqref{Aas}, \eqref{vN2}, we may assert that
\beq\label{b1}
B(v_N(t),u(t))=\mathcal O_{\al,\s}(e^{-(N+1+\varep_*)t}).
\eeq

On the other hand, for $1\le n\le N$, by \eqref{als} we have 
\beq \label{oun}
u_n(t)=q_n(t)e^{-nt}=\mathcal O_{\al+1/2,\s}(e^{-(n-\delta_*)t})\quad \text{with } \delta_*=1/4.
\eeq
In particular, $u_n(t)=\mathcal O_{\al+1/2,\s}(e^{-t})$ for $2\le n\le N$.
Since, $u_1(t)$ is already $\mathcal O_{\al+1/2,\s}(e^{-t})$, then 
\beq\label{uNo} 
\bar u_N(t)=\mathcal O_{\al+1/2,\s}(e^{-t}).
\eeq 
Applying \eqref{AalphaB} again and using \eqref{vN2}, \eqref{uNo} yield  
\beq\label{b2} 
B(\bar u_N(t),v_N(t))=\mathcal O_{\al,\s}(e^{-(N+1+\varep_*)t}).
\eeq

Lastly, observe that by \eqref{oun} and inequality \eqref{AalphaB} we have
$B(u_m,u_j)=\mathcal O_{\al,\s}(e^{-(m+j-2\delta_*)t})$.
Thus, for $m+j\ge N+2$ we have
\beq\label{b3}
B(u_m,u_j)=\mathcal O_{\al,\s}(e^{-(N+2-1/2)t})=\mathcal O_{\al,\s}(e^{-(N+1+\varep_*)t}).
\eeq

Combining definition \eqref{hdef} of $h_N$ with \eqref{b1}, \eqref{b2} and \eqref{b3}
gives the desired \eqref{hNo}.

\medskip
\noindent\textit{II.3.~Construction of $q_{N+1}$.}
Observe that
\beq
	v_N=\sum_{k\leq N}R_kv_N+R_{N+1}v_N+\sum_{k\geq N+2}R_{k}v_N.\notag
\eeq
By \eqref{vNeq}, we have
\beqs
\frac d{dt}R_kv_N+kR_kv_N +\sum_{m+j=N+1}R_kB(u_m,u_j)=R_kh_N,
\eeqs
or equivalently, by definition of $u_j$ we have
\beq\label{Rvn:alt}
\frac d{dt}R_kv_N+kR_kv_N +\sum_{m+j=N+1}e^{-(N+1)t}R_kB(q_m, q_j)=R_kh_N.
\eeq
We will extract polynomials associated with the correct exponential decay from each regime: $k\leq N$, $k=N+1$, and $k\geq N+2$.  In each case, we show that their error from $v_N$ are of the desired order, i.e., \eqref{err}.

{\flushleft \textbf{Case $k=N+1$.} }  By \eqref{Rvn:alt}, we have
\beqs
\frac d{dt}R_{N+1}v_N+(N+1)R_{N+1}v_N +\sum_{m+j=N+1} e^{-(N+1)t}R_{N+1}B(q_m,q_j) =R_{N+1}h_N(t).
\eeqs
With the integrating factor $e^{(N+1)t}$, we obtain
\begin{align}\label{er0}
	\begin{split}
e^{(N+1)t}&R_{N+1}v_N(t)\\
		&=R_{N+1}v_N(0)-\sum_{m+j=N+1}\int_0^tR_{N+1}B(q_m(\tau), q_j(\tau))\ d\tau+\int_0^te^{(N+1)\tau}h_N(\tau)\ d\tau.
	\end{split}
\end{align}

By \eqref{hNo},
\beq\label{erh}
e^{(N+1)t}R_{N+1} h_N(t)=\mathcal O_{\al,\s}(e^{-\varep_* t}).
\eeq
This and the fact $R_{N+1}h_N(t)$ is continuous on $[0,\infty)$  imply 
\beq\label{xi} 
\xi_{N+1}\eqdef R_{N+1}v_N(0)+\int_0^\infty e^{(N+1)\tau} R_{N+1}h_N(\tau)\ d\tau \quad \text{ exists and belongs to }  R_{N+1}H.
\eeq
Thus, let us define the polynomial
\beq \label{p1}
p_{N+1,N+1}(t)\eqdef\xi_{N+1}-\sum_{m+j=N+1}\int_0^tR_{N+1}B(q_m(\tau),q_j(\tau))\ d\tau.
\eeq

Then by \eqref{erh} we have
\begin{align*}
e^{(N+1)t}R_{N+1}v_N(t)-p_{N+1,N+1}(t)
&=-\int_t^\infty e^{(N+1)\tau} R_{N+1}h_N(\tau)d\tau=\mathcal O_{\al,\s}(e^{-\varep_* t}),
\end{align*}
and hence, that
\beq\label{k1}
R_{N+1}v_N(t)-e^{-(N+1)t}p_{N+1,N+1}(t)=\mathcal O_{\al,\s}(e^{-(N+1+\varep_*) t}).
\eeq

{\flushleft\textbf{Case $k\le N$.}}  Firstly, with the integrating factor $e^{kt}$ and the fact, by \eqref{vN2}, 
$$\lim_{t\to\infty}e^{kt}R_kv_N(t)=0,$$
we obtain from \eqref{Rvn:alt} that
\beq\label{form1}
e^{kt}R_kv_N(t)
=\int_t^\infty e^{(k-N-1)\tau}\Big(\sum_{m+j=N+1} R_kB(q_m,q_j) \Big)d\tau-\int_t^\infty e^{k\tau}R_k h_N(\tau))d\tau. 
\eeq

Observe that by property \eqref{hNp}, we estimate for large $t$ that
\begin{align}
\Big|\int_t^\infty e^{k\tau}R_k h_N(\tau)d\tau\Big|_{\al,\s}
&\le H_{N,\al}\int_t^\infty e^{-(N+1+\varep_*-k)\tau}d\tau=\frac{H_{N,\al}e^{-(N+1+\varep_*-k)t}}{N+1+\varep_*-k} \notag \\
&\le 2H_{N,\al}e^{-(N+1+\varep_*-k)t}. \label{hleN}
\end{align}

An elementary calculation shows that for any $\beta>0$ and integer $d\ge 0$ one has
\beqs
\int_t^\infty  \tau^d e^{-\beta \tau}\ d\tau=e^{-\be t}\sum_{n=0}^{d}\frac{d!}{n!\beta^{d+1-n}}t^n.
\eeqs
Thus
\beq\label{i1}
	p_{N+1,k}(t)\eqdef e^{(N+1-k)t}\sum_{m+j=N+1}\int_t^\infty e^{-(N+1-k)\tau}R_kB(q_m(\tau), q_j(\tau))\ d\tau
\eeq
defines a polynomial in $t$, valued in $R_kH$. 

Returning then to \eqref{form1}, we apply \eqref{hleN} and \eqref{i1} to derive for large $t$ that 
\begin{align*}
|e^{kt}R_kv_N(t)-e^{(k-N-1)t}p_{N+1,k}(t)|_{\al,\s}\le 2H_{N,\al} e^{-(N+1+\varep_*-k)t}.
\end{align*}
Multiplying this inequality by $e^{-kt}$ gives
\beqs
R_kv_N(t)-e^{-(N+1)t}p_{N+1,k}(t) =\mathcal O_{\al,\s}(e^{-(N+1+\varep_*)t}).
\eeqs
Summing this identity over $1\leq k\leq N$ we obtain
\beq\label{k2}
\sum_{k=1}^NR_kv_N(t)-\sum_{k=1}^N e^{-(N+1)t}p_{N+1,k}(t) =\mathcal O_{\al,\s}(e^{-(N+1+\varep_*)t}).
\eeq

{\flushleft\textbf{Case $k\ge N+2$.} }
Let $T_\al$ be as in \eqref{hNp}.
For $t>T_\al$, we have from \eqref{Rvn:alt} that
\begin{align}\label{bq}
	\begin{split}
R_k& v_N(t)=e^{-k(t-T_\al)} R_k v_N(T_\al)\\
	&-\sum_{m+j=N+1}e^{-kt}\int_{T_\al}^t e^{(k-(N+1))\tau } R_kB(q_m(\tau),q_j(\tau))\ d\tau
+\int_{T_\al}^t e^{-k(t-\tau)}R_k h_N(\tau) d\tau.
	\end{split}
\end{align}

Consider the first integral on the right-hand side of \eqref{bq}.
An elementary calculation shows that for any integer $d\ge 0$ and  $\beta\in\R$ nonzero we have 
\beqs
\int  t^d e^{\beta t}\ dt=e^{\beta t}\sum_{n=0}^{d}\frac{(-1)^{d-n} d!}{n!\beta^{d+1-n}}t^n +const.
	\eeqs
This identity and the fact that each $B(q_m(\tau),q_j(\tau))$ is a polynomial of $\tau$ imply that
there is a polynomial, $p_{N+1,k}(t)$, in $t$, valued in $R_kH$ such that
\begin{multline}\label{i2}
-\int_{T_\al}^t e^{(k-(N+1))\tau }\Big(\sum_{m+j=N+1} R_kB(q_m,q_j) \Big) d\tau \\
= e^{(k-(N+1))t}p_{N+1,k}(t)-e^{(k-(N+1))T_\al}p_{N+1,k}(T_\al).
\end{multline}

Then  \eqref{bq}  gives
\begin{multline}\label{klarge}
R_k v_N(t)=e^{-k(t-T_\al)} R_k v_N(T_\al)\\
 + \Big(e^{-(N+1)t}p_{N+1,k}(t)-e^{-k(t-T_\al) -(N+1)T_\al}p_{N+1,k}(T_\al)\Big)
+\int_{T_\al}^t e^{-k(t-\tau)}R_k h_N(\tau) d\tau.
\end{multline}

It follows from \eqref{klarge} and \eqref{hNp} that
\begin{align*}
&|R_k v_N(t)-e^{-(N+1)t}p_{N+1,k}(t)|_{\al,\s}\\
&\le e^{-k(t-T_\al)}\Big( |R_k v_N(T_\al)|_{\al,\s} + e^{-(N+1)T_\al}|p_{N+1,k}(T_\al)|_{\al,\s}\Big)
+\int_{T_\al}^t e^{-k(t-\tau)}|h_N(\tau)|_{\al,\s}  d\tau\\
&\le e^{-(N+2)(t-T_\al)}\Big( |R_k v_N(T_\al)|_{\al,\s} + |p_{N+1,k}(T_\al)|_{\al,\s}\Big)
+\int_{T_\al}^t e^{-k(t-\tau)}H_{N,\al} e^{-(N+1+\varep_*)\tau}  d\tau.
\end{align*}
Elementary calculations give
\begin{align*}
|R_k v_N(t)-e^{-(N+1)t}p_{N+1,k}(t)|_{\al,\s}
&\le e^{-(N+2)(t-T_\al)}( |R_k v_N(T_\al)|_{\al,\s}+ |p_{N+1,k}(T_\al)|_{\al,\s})\\
&\quad +\frac{H_{N,\al} e^{-(N+1+\varep_*)t}}{k-(N+1+\varep_*)} .
\end{align*}

Squaring the preceding inequality, using the 3-term Cauchy-Schwarz inequality for the right-hand side, and then summing up in $k$ yield
\begin{align}
&\sum_{k\ge N+2} |R_k v_N(t)-e^{-(N+1)t}p_{N+1,k}(t)|_{\al,\s}^2 \notag 
\le 3e^{-2(N+2)(t-T_\al)}\\
&\quad \cdot \Big( \sum_{k\ge N+2}|R_k v_N(T_\al)|_{\al,\s}^2 + \sum_{ N+2\le k\le s_{N+1}}|p_{N+1,k}(T_\al)|_{\al,\s}^2\Big) +\sum_{k\ge N+2} \frac{3H_{N,\al}^2 e^{-2(N+1+\varep_*)t}}{(k-(N+1+\varep_*))^2} \notag \\
&\le e^{-2(N+2)t}E_1^2 +  e^{-2(N+1+\varep_*)t} E_2^2, \label{QN1}
\end{align}
where
\beqs
E_1^2= 3e^{2(N+2)T_\al}\big ( |v_N(T_\al)|_{\al,\s}^2+ \sum_{N+2\le k\le s_{N+1}}|p_{N+1,k}(T_\al)|_{\al,\s}^2\big ),
\eeqs
\beqs
E_2^2=\sum_{k\ge N+2} \frac{3H_{N,\al}^2}{(k-(N+1+\varep_*))^2}=3H_{N,\al}^2\sum_{k\ge 1} \frac{1}{(k-1/2)^2}.
\eeqs

{\flushleft \textbf{Definition of $q_{N+1}$.}}  Finally, we define the polynomial
\beq\label{qdef}
q_{N+1}(t)\eqdef\sum_{1\le k\le s_{N+1}}p_{N+1,k}(t),\text{ that is, } R_kq_{N+1}(t)=p_{N+1,k}(t)
\text{ for }k=1,\dots, s_{N+1},
\eeq
where $p_{N+1,k}(t)$ are polynomials defined in \eqref{p1}, \eqref{i1} and \eqref{i2}.  

Therefore, with  $q_{N+1}(t)$ defined as in \eqref{qdef}, properties \eqref{k1} and \eqref{k2} can be summarized as
\beq\label{k5}
|P_{N+1}(v_N(t)-e^{-(N+1)t}q_{N+1}(t))|_{\al,\s}
= \mathcal O(e^{-(N+1+\varep_*)t}),
\eeq
while \eqref{QN1} can be equivalently expressed as
\beq\label{k4}
|(I-P_{N+1})(v_N(t)-e^{-(N+1)t}q_{N+1}(t))|_{\al,\s}^2
= \mathcal O(e^{-2(N+1+\varep_*)t}).
\eeq

Let $u_{N+1}(t)\eqdef e^{-(N+1)t}q_{N+1}(t)\in P_{s_{N+1}}H$ and $v_{N+1}\eqdef v_N-u_{N+1}$. Then \eqref{k4} and \eqref{k5} imply
\beq\label{k3}
|v_{N+1}(t)|_{\al,\s}=\mathcal O(e^{-(N+1+\varep_*)t}).
\eeq
Since $v_{N+1}=u-\sum_{n=1}^{N+1}u_n$, inequality \eqref{k3}  proves \eqref{err} for $N+1$.

\medskip
\noindent\textit{II.4.~Evolution of $u_{N+1}$.} 
To complete the induction step, it remains to show that $u_{N+1}$ satisfies the ODE \eqref{um:eqn}.  Observe that we need only to show that the equation holds in the finite dimensional space $P_{s_{N+1}}H$, or equivalently, in $R_kH$ for $1\le k\le s_{N+1}$.

For $1\le k\le N$ and $N+2\leq k\leq s_{N+1}$ we have from \eqref{i1} and \eqref{i2}, respectively, that
	\begin{align*}
&\frac{d}{dt}R_ku_{N+1}(t) + AR_ku_{N+1}(t)
=\frac{d}{dt}R_ku_{N+1}(t) + kR_ku_{N+1}(t) \\
&=e^{-kt} \frac{d}{dt} \big(e^{kt}R_ku_{N+1}(t)\big)=e^{-kt} \frac{d}{dt}\big(e^{(k-N-1)t}p_{N+1,k}(t)\big)\\
&=e^{-kt}\Big[ -e^{(k-N-1)t}\sum_{m+j=N+1} R_kB(q_m(t),q_j(t))\Big] =-\sum_{m+j=N+1} R_kB(u_m(t),u_j(t)) .
	\end{align*}

For $k=N+1$, we have from \eqref{p1} that
	\begin{align*}
		\frac{d}{dt}R_{N+1}u_{N+1}(t)&=\frac{d}{dt}\left(e^{-(N+1)t}p_{N+1,N+1}(t)\right)\notag\\
					&=-(N+1)e^{-(N+1)t}p_{N+1,N+1}(t)-e^{-(N+1)t}\sum_{m+j=N+1}R_{N+1}B(q_m(t),q_j(t)\\
					&=-AR_{N+1}u_{N+1}(t)-\sum_{m+j=N+1}R_{N+1}B(u_m(t),u_j(t)).
	\end{align*}
Hence
	\begin{align*}
		\frac{d}{dt}R_ku_{N+1}+AR_ku_{N+1}+\sum_{m+j=N+1}R_kB(u_m,u_j)=0
	\end{align*}
holds for each $1\leq k\leq s_{N+1}$ and therefore, that \eqref{um:eqn} holds for $u_{N+1}$ in $P_{s_{N+1}}H$.

Since \eqref{err} and \eqref{um:eqn} hold for $n=N+1$, this completes the induction step.  Thus, ($\mathcal T_N$) is true for all $N\in\N$.
In concluding this proof, let us remark that in the induction step, the polynomials, $q_n(t)$, $n=1,\dots, N$ appearing in  ($\mathcal T_{N+1}$) are precisely those from ($\mathcal T_N$). Hence, the polynomials $q_n(t)$'s exist for all $n\in \N$.
\end{proof}

We are ready to prove the main result.

\begin{proof}[Proof of Theorem \ref{mainthm}]
Let $u(t)$ be a Leray-Hopf weak solution. For any $\s>0$, Theorem \ref{decay2} and Lemma \ref{STsigma} imply that there is $T_\s>0$ such that $u(T_\s)\in\solnset$, solution $u(t)$ is regular on $[T_\s,\infty)$ and there are polynomials $Q_n^\s(t)$ for all $n\in\N$ such that $u^\s(t)\eqdef u(T_\s+t)$ satisfies for each $N\ge 1$ and all $\al>0$ that
the expansion
\beq\label{x1}
u^\s(t)\sim \sum_{n=1}^\infty Q^\s_n(t)e^{-nt} \text{ as }t\to\infty  \text{ holds in } G_{\al,\s}
\eeq
with 
\beqs
\Big| u^\s(t)-\sum_{n=1}^N Q^\s_n(t)e^{-nt} \Big|_{\al,\s} =\mathcal O(e^{-(N+1/2)t})\text{ as } t\to\infty.
\eeqs

By defining 
\beq\label{qsig}
q_n^\s(t)=Q_n^\s(t-T_\s) e^{nT_\s},
\eeq
we have for any $N\ge 1$ and $\al>0$ that
\beq\label{remain1}
\begin{aligned}
\Big| u(t)-\sum_{n=1}^N q^\s_n(t)e^{-nt} \Big|_{\al,\s} 
&=\Big| u^\s(t-T_\s)-\sum_{n=1}^N Q^\s_n(t-T_\s)e^{-n(t-T_\s)} \Big|_{\al,\s}\\
& =\mathcal O(e^{-(N+1/2)(t-T_\s)})
=\mathcal O(e^{-(N+1/2)t})\quad \text{as } t\to\infty.
\end{aligned}
\eeq

It remains to prove that $q_n^\s(t)$ is independent of $\s$.
Suppose $\s'\ne \s$ and $T_{\s'}\ge T_\s$.  Then applying \eqref{x1} to $\s'$ in place of $\s$ gives
\beq\label{x2}
u(T_{\s'}+t)\sim \sum_{n=1}^\infty Q^{\s'}_n(t)e^{-nt}
\eeq
and, at the same time, from \eqref{x1}
\beq\label{x3}
\begin{aligned}
u(T_{\s'}+t)=u^\s(t+T_{\s'}-T_\s)
&\sim \sum_{n=1}^\infty Q^\s_n(t+T_{\s'}-T_\s)e^{-n(t+T_{\s'}-T_\s)}\\
&=\sum_{n=1}^\infty Q^\s_n(t+T_{\s'}-T_\s)e^{-n(T_{\s'}-T_\s)}e^{-nt}.
\end{aligned}
\eeq

Since both \eqref{x2} and \eqref{x3} can be seen as asymptotic expansions in $H$ for the regular solution $u(T_{\s'}+t)$ with $t\ge 0$, by the expansion's uniqueness, we have
\beq\label{QQ}
Q^{\s'}_n(t)=Q^\s_n(t+T_{\s'}-T_\s)e^{-n(T_{\s'}-T_\s)}.
\eeq
Then it follows definition \eqref{qsig} and relation \eqref{QQ} that
\beqs
q_n^{\s'}(t)=Q_n^{\s'}(t-T_{\s'}) e^{nT_{\s'}}
=Q^\s_n(t-T_\s)e^{-n(T_{\s'}-T_\s)}  e^{nT_{\s'}}
=Q^\s_n(t-T_\s)e^{nT_\s} =q_n^\s(t).
\eeqs

For $n\ge 1$, let $q_n(t)=q_n^1(t)$ which is defined by \eqref{qsig} with $\s=1$.
Then $q_n^\s=q_n$ for all $n\ge 1$ and $\s>0$.
Therefore \eqref{remain1} holds for all $N\ge 1$ and  $\al,\s>0$.
Note for all $N\ge 1$ and  $\al,\s>0$ that, with the same notation used in Lemma \ref{STsigma}, as $t\to\infty$ 
$$v_N(t)=q_{N+1}(t)e^{-(N+1)t}+v_{N+1}(t)
=\mathcal O_{\al,\s}(e^{-(N+\varep)t})+\mathcal O_{\al,\s}(e^{-(N+3/2)t})=\mathcal O_{\al,\s}(e^{-(N+\varep)t}),$$
for any $\varep\in(0,1)$, which yields  \eqref{remain}.
The proof of Theorem \ref{mainthm} is complete.
\end{proof}

\begin{remark}
By combining this paper's method with those in \cite{FLOZ1,FLOZ2,FLS1}, we can study the associated normal form to the expansion \eqref{expand}, and its solutions  in the Gevrey spaces.  This study will be pursued in a subsequent work.
\end{remark}

\appendix

\section{}\label{apex}

\begin{proof}[Proof of Lemma \ref{nonLem}] 
The proof follows Foias-Temam \cite{FT-Gevrey} and the Sobolev-norm version in \cite[Lemma 2.3]{FLS1}.
Let $u,v,w$ be $H$ with 
\beqs u=\sum_{\veck\ne 0} \hat \vecu(\veck)e^{-i\veck\cdot \vecx},\ 
v=\sum_{\veck\ne 0} \hat \vecv(\veck)e^{-i\veck\cdot \vecx},\ w=\sum_{\veck\ne 
0} \hat \vecw(\veck)e^{-i\veck\cdot \vecx}.\eeqs 
Define the scalar functions
\beq \label{ustar}
u_*=\sum_{\veck\ne 0} |\hat \vecu(\veck)|e^{-i\veck\cdot \vecx},\ 
v_*=\sum_{\veck\ne 0} |\hat \vecv(k)|e^{-i\veck\cdot \vecx},\ w_*=\sum_{\veck\ne 
0} |\hat \vecw(\veck)|e^{-i\veck\cdot \vecx}.\eeq 
Then 
\beq\label{starrelation} |A^\alpha u|=|(-\Delta)^{\alpha} u_*| \ \hbox{for all}\ 
\alpha\ge 0.\eeq
We have
\begin{align*}
\inprod{A^\alpha e^{\sigma A^{1/2}}B(u,v),w}
&=8\pi^3\sum_{\veck+\vecl+\vecm=0} |\vecm|^{2\alpha}e^{\sigma |\vecm|} (\hat 
\vecu(\veck)\cdot \vecl)\, ( \hat \vecv(\vecl) \cdot \hat \vecw(\vecm) ).
\end{align*}
Since  
$$|\vecm|^{2\alpha}=|\veck+\vecl|^{2\alpha} \le 2^{2\alpha}(|\veck|^{2\alpha}+|\vecl|^{2\alpha})
\quad
\text{and}\quad
e^{\sigma |\vecm|}\le e^{\sigma |\veck|} e^{\sigma |\vecl|},
$$
it follows that
\begin{align*}
\lvert\inprod{A^\alpha e^{\sigma A^{1/2}}B(u,v),w}\rvert
&\le 8\pi^3 4^\alpha \sum_{\veck+\vecl+\vecm=0} |\veck|^{2\alpha}e^{\sigma 
|\veck|}  |\hat \vecu(\veck)|\cdot e^{\sigma |\vecl|}  |\vecl|\cdot | \hat \vecv(\vecl) 
| |\hat \vecw(\vecm) |\\
&\quad + 8\pi^3 4^\alpha \sum_{\veck+\vecl+\vecm=0} e^{\sigma |\veck|}  |\hat 
\vecu(\veck)|\cdot |\vecl|^{2\alpha+1}e^{\sigma |\vecl|} \cdot | \hat \vecv(\vecl) | 
|\hat \vecw(\vecm) |.
\end{align*}
Rewriting the last inequality's right-hand side in terms of $u_*$, $v_*$ and $w_*$ gives
\begin{multline}\label{starint} 
\lvert\inprod{A^\alpha e^{\sigma A^{1/2}}B(u,v),w}\rvert
\le 8\pi^34^\al\left|\int_{\Omega} ((-\Delta)^{\alpha}e^{\sigma 
A^{1/2}}u_*) \cdot ((-\Delta)^{1/2}e^{\sigma A^{1/2}}v_*) \cdot w_*\ dx\right|\\
+8\pi^34^\al\left|\int_{\Omega} (e^{\sigma 
A^{1/2}}u_*) \cdot ((-\Delta)^{\al+1/2}e^{\sigma A^{1/2}}v_*) \cdot w_*\ dx\right|
\eqdef 8\pi^34^\al I_1+8\pi^34^\al I_2.
\end{multline}

We recall the Sobolev, interpolation, and Agmon inequalities for functions $u_*$, $v_*$, $w_*$ of the form in \eqref{ustar}. There are   positive constants $c_1$ and $c_2$ such that  
\beqs
\|u_*\|_{L^6(\Omega)}\le c_1 |(-\Delta)^{1/2}u_*|,
\eeqs
\beqs
\|u_*\|_{L^3(\Omega)}\le c_1^{1/2} |(-\Delta)^{1/4}u_*|,
\eeqs
\beqs
\|u_*\|_{L^\infty(\Omega)}\le c_2|(-\Delta)^{1/2}u_*|^{1/2} |(-\Delta)u_*|^{1/2}.
\eeqs

For $I_1$ in \eqref{starint}, we apply H\"older's inequality with powers $3$, $6$, and $2$.  Then by the above interpolation and  Sobolev inequalities, and relation \eqref{starrelation}, we obtain 
\begin{align}
I_1
&\le  \|(-\Delta)^{\alpha}e^{\sigma 
A^{1/2}}u_*\|_{L^3(\Omega)} \|(-\Delta)^{1/2}e^{\sigma A^{1/2}}v_*\|_{L^6(\Omega)} | w_*| \notag\\
&\le c_1^{3/2}|(-\Delta)^{\alpha+1/4}e^{\sigma A^{1/2}}u_*| |(-\Delta)e^{\sigma 
A^{1/2}}v_*| |w_*| \notag\\
&\le c_1^{3/2}|A^{\alpha+1/4}e^{\sigma A^{1/2}}u| |Ae^{\sigma A^{1/2}}v||w|.\label{t1}
\end{align}

Similarly, estimating $I_1$ by H\"older's inequality with powers $6,3,2$, and then using interpolation inequalities and the relation \eqref{starrelation}, we obtain
\begin{align}
 I_1
&\le  \|(-\Delta)^{\alpha}e^{\sigma 
A^{1/2}}u_*\|_{L^6(\Omega)} \|(-\Delta)^{1/2}e^{\sigma A^{1/2}}v_*\|_{L^3(\Omega)} | w_*| \notag\\
 & \le c_1^{3/2}|(-\Delta)^{\alpha+1/2}e^{\sigma A^{1/2}}u_*| 
|(-\Delta)^{3/4}e^{\sigma A^{1/2}}v_*| |w_*| \notag\\
&\le  c_1^{3/2}|A^{\alpha+1/2}e^{\sigma A^{1/2}}u| |A^{3/4}e^{\sigma A^{1/2}}v| 
|w|. \label{t2}
\end{align}

For $I_2$ in  \eqref{starint}, applying the H\"older 
inequality  and then using the Agmon inequality for the embedding of $\mD(A)$ 
into $L^\infty(\Omega)^3$,  we obtain
\begin{align}
 I_2&
\le \|e^{\sigma A^{1/2}}u_*\|_{L^\infty(\Omega)} 
|(-\Delta)^{\alpha+1/2}e^{\sigma A^{1/2}}v_*||w_*| \notag \\
&\le c_2 |(-\Delta)^{1/2}e^{\sigma A^{1/2}}u_*|^{1/2}|(-\Delta)e^{\sigma 
A^{1/2}}u_*|^{1/2} \cdot |(-\Delta)^{\alpha+1/2}e^{\sigma A^{1/2}}v_*||w_*| \notag \\
&\le c_2|A^{1/2}e^{\sigma A^{1/2}}u|^{1/2}|Ae^{\sigma 
A^{1/2}}u|^{1/2}|A^{\alpha+1/2}e^{\sigma A^{1/2}}v||w|. \label{t3}
\end{align}

Combining \eqref{starint} with \eqref{t3} and \eqref{t1}, resp.~\eqref{t2}, 
yields \eqref{B1}, resp.~\eqref{B2} with
\beqs\label{C-const} 
c_*=8\pi^3\max\{c_1^{3/2},c_2\}.\eeqs

For $\alpha\ge 1/2$, it follows either \eqref{B1} or \eqref{B2} that
\beqs |A^\alpha e^{\sigma A^{1/2}} B(u,v)|\le 2c_* 4^\alpha  |A^{\alpha+1/2}e^{\sigma A^{1/2}}u| \, 
|A^{\alpha+1/2}e^{\sigma A^{1/2}}v|,  
\eeqs
and, hence, we obtain \eqref{AalphaB}.  
\end{proof}

\medskip
\noindent\textbf{Acknowledgement.} The authors would like to thank Ciprian Foias and Edriss S. Titi for insightful discussions.
L.H. acknowledges the support by NSF grant DMS--1412796.

\def\cprime{$'$}

\end{document}